\documentclass[10pt]{article}
\usepackage[a4paper,centering,hscale=0.7,vscale=0.7]{geometry}

\usepackage{lmodern}
\usepackage[T1]{fontenc}

\usepackage{mathtools} 
\usepackage{amsthm,amssymb}
\usepackage{mathrsfs}

\usepackage{hyperref}

\newtheorem{prop}{Proposition}[section]
\newtheorem{thm}[prop]{Theorem}
\newtheorem{lemma}[prop]{Lemma}
\newtheorem{coroll}[prop]{Corollary}

\theoremstyle{definition}

\theoremstyle{remark}
\newtheorem{rmk}[prop]{Remark}

\newcommand{\bD}{\mathbb{D}}
\newcommand{\E}{\mathop{{}\mathbb{E}}\nolimits}
\newcommand{\cF}{\mathscr{F}}
\newcommand{\cL}{\mathscr{L}}
\newcommand{\enne}{\mathbb{N}}
\renewcommand{\P}{\mathbb{P}}
\newcommand{\erre}{\mathbb{R}}

\DeclareMathAlphabet{\mathbbmsl}{U}{bbm}{m}{sl}
\newcommand{\bL}{\mathbbmsl{L}}

\newcommand{\embed}{\hookrightarrow}
\renewcommand{\geq}{\geqslant}
\renewcommand{\leq}{\leqslant}
\newcommand{\loc}{\textnormal{loc}}

\numberwithin{equation}{section}

\DeclarePairedDelimiter\abs{\lvert}{\rvert}
\DeclarePairedDelimiter\norm{\lVert}{\rVert}
\DeclarePairedDelimiterX\ip[2]{\langle}{\rangle}{#1,#2}

\DeclarePairedDelimiterX\cc[2]{[\![}{]\!]}{#1,#2}

\mathtoolsset{centercolon}

\makeatletter
\renewenvironment{enumerate}{%
  \ifnum \@enumdepth >3 \@toodeep\else
      \advance\@enumdepth \@ne
      \edef\@enumctr{enum\romannumeral\the\@enumdepth}\list
      {\csname label\@enumctr\endcsname}{\usecounter
        {\@enumctr}\def\makelabel##1{\hss\llap{\upshape##1}}}\fi
}{%
  \endlist
}

\renewenvironment{itemize}{%
  \ifnum\@itemdepth>3 \@toodeep
  \else \advance\@itemdepth\@ne
    \edef\@itemitem{labelitem\romannumeral\the\@itemdepth}%
    \list{\csname\@itemitem\endcsname}%
      {\def\makelabel##1{\hss\llap{\upshape##1}}}%
  \fi
}{%
  \endlist
}

\makeatother

\begin{document}
\begin{center}
  {\Large\bfseries On pointwise Malliavin differentiability of
    solutions to semilinear parabolic SPDEs \par}%
  \vskip 1.5em%
  {
    \lineskip .5em%
    Carlo Marinelli \par}%
  \vskip 1em%
  {December 31, 2021}%
\end{center}
\par
\vskip 1em

\begin{abstract}
  We obtain estimates on the first-order Malliavin derivative of mild
  solutions, evaluated at fixed points in time and space, to a class
  of parabolic dissipative stochastic PDEs on bounded domain of
  $\mathbb{R}^d$. In particular, such equations are driven by
  multiplicative Wiener noise and the nonlinear drift term is the
  superposition operator associated to a locally Lipschitz continuous
  function satisfying suitable polynomial growth bounds. The main
  arguments rely on the well-posedness theory in the mild sense for
  stochastic evolution equations in Banach spaces, monotonicity,
  and a comparison principle.
\end{abstract}


\section{Introduction}
Consider the stochastic evolution equation
\begin{equation}
  \label{eq:0}
  du + Au\,dt = f(u)\,dt + \sigma(u)B\,dW(t),
  \qquad u(0)=u_0,
\end{equation}
where $A$ is the negative generator of an analytic semigroup of
contractions $S$ on $L^q(G)$, with $q \geq 2$ and $G \subset \erre^d$
a smooth bounded domain; $f \colon \erre \to \erre$ is a decreasing
locally Lipschitz continuous function such that $\abs{f(x)} \lesssim 1
+ \abs{x}^m$; $\sigma\colon \erre \to \erre$ is a Lipschitz continuous
function; $W$ is a cylindrical Wiener process on a separable Hilbert
space $U$, and $B$ is $\gamma$-Radonifying from $U$ to
$L^q(G)$. Precise assumptions are given in \S\ref{ssec:ass} below. 

In \cite{cm:POTA21} we proved that the unique mild solution to
\eqref{eq:0}, which is continuous and time and space under mild extra
assumptions, is such that the law of the random variable $u(t,x)$ is
absolutely continuous with respect to Lebesgue measure for every
$(t,x) \in \mathopen]0,T] \times G$. This was achieved considering
first the case where $f$ is Lipschitz continuous, then using a
localization argument implying that $u(t,x) \in \bD^{1,p}_\loc$, hence
applying the well-known Bouleau-Hirsch criterion. This reasoning, by
its very nature, does not allow to show that $u(t,x)$ belongs to any
Malliavin space. On the other hand, for equations with additive noise,
i.e. for which $\sigma$ does not depend on $u$, it was proved in
\cite{cm:POTA13} that $u(t,x)$ belongs to $\bD^{1,p}$ for all $p \geq
1$, and even to $\bD^{k,p}$ is the coefficients $f$ and $\sigma$ are
of class $C^k$ with all derivatives satisfying polynomial
bounds. Equations driven by additive noise are in fact much easier to
treat because the corresponding equations for Malliavin derivatives
are deterministic PDEs with random coefficients, for which a large
number of analytic tools can be applied pathwise.
In the case of equations driven by multiplicative noise it is
unfortunately impossible to follow this route, as the Malliavin
derivatives of solutions satisfy linear stochastic PDEs of rather
unfriendly character: for instance, the initial condition contains a
Dirac measure in time. 

Our goal is to fill at least in part the gap between the results of
\cite{cm:POTA21} and those of \cite{cm:POTA13}, showing that $u(t,x)$
belongs to $\bD^{1,p}$ for all $p \geq 1$. To this purpose, we develop
two different approaches: one uses stochastic calculus in
vector-valued $L^q$ spaces and monotonicity, and another one a
comparison principle for mild solutions to stochastic evolution
equations. In the former approach we need a kind of smoothness (or
boundedness) assumption on the noise, while in the latter the
covariance of the noise is assumed to be a positivity-preserving
operator (such an assumption was used in \cite{cm:POTA13} as well).

We refer to the introduction of \cite{cm:POTA13} for references to the
(not very extensive) literature on the pointwise Malliavin
differentiability of solutions to equations with coefficients growing
faster than linearly, as well as for a short discussion on potential
applications.

\medskip\par\noindent
\textbf{Acknowledgments.} The hospitality of the Interdisziplin\"ares
Zentrum f\"ur Komplexe Systeme, Universit\"at Bonn, Germany, is
gratefully acknowledged.


\section{Preliminaries}
\subsection{Assumptions and notation}
\label{ssec:ass}
The linear unbounded operator $A$ appearing in \eqref{eq:0} is
supposed to be the (negative) generator of a strongly continuous
analytic semigroup $S$ on $L^2(G)$, with $G$ a smooth bounded domain
of $\erre^d$. Furthermore, we assume that $S$ is self-adjoint and
Markovian. Then $S$ can be restricted in a consistent manner to
analytic semigroups, necessarily of contractions, to all $L^q:=L^q(G)$
spaces, $q \geq 2$. Moreover, $S$ can be extended in a unique way from
$L^q$ to an analytic contraction semigroup on $L^q(G;H)$
(see~\cite{Pis:hol}). Finally, we assume that $S$ admits a kernel,
i.e. that there exists a measurable function
$K\colon \erre_+ \times G^2 \to \erre_+$ such that
\[
  \bigl[ S(t)\phi \bigr](x) = \int_G K_t(x,y)\phi(y)\,dy
\]
for every $\phi \in L^2$.

The function $f\colon \erre \to \erre$ is assumed to be continuously
differentiable, decreasing, and such that
\[
  \abs{f(x)} + \abs{f'(x)} \lesssim 1 + \abs{x}^m \qquad \forall x \in
  \erre
\]
for some $m \in \erre_+$. More generally, it would be enough to assume
that $x \mapsto f(x) - a x$ is decreasing, for some $a \in \erre_+$,
and that $f$ is just almost everywhere differentiable. This little
extra generality causes too much notational and technical nuisance to
be justified.

The function $\sigma\colon \erre \to \erre$ is of class $C^1$ with
bounded derivative. The more general case of $\sigma$ being just local
Lipschitz continuous with linear growth could be treated as well,
again at the cost of some nuisance.

$W$ is a cylindrical Wiener process on a separable Hilbert space $U$
defined on a filtered probability space
$(\Omega,\cF,(\cF_t)_{t \in [0,T]},\P)$, with $T \in \erre_+$, where
$(\cF_t)_{t \in [0,T]}$ is the completion of the filtration generated
by $W$. All random quantities are assumed to be supported on this
stochastic basis.

The operator $B$ is assumed to belong to $\gamma(U;L^q)$, the space of
$\gamma$-Radonifying operators from $U$ to $L^q$. The precise value of
$q$ will be in the range needed for Proposition~\ref{prop:0} below to
hold. Additional assumptions on $B$ will be stated when needed.

The initial datum $u_0$ is assumed to belong to $C(\overline{G})$,
although less is needed for just existence of a mild solution to
\eqref{eq:0}.

\smallskip

The Hilbert space on which the Malliavin calculus will be based is
$L^2(0,T;L^2_Q)$. Here $Q:=BB^*$ is a symmetric trace-class operator
on $L^2$, in particular $Q^{1/2}$ exists and is itself a bounded
operator, for which assume that $\ker Q^{1/2} = \{0\}$. Then $L^2_Q$
is the completion of $L^2$ with respect to the norm induced by the
scalar product
\[
  \ip{f}{g}_Q:=\ip{Qf}{g}_{L^2(G)} = \ip{Q^{1/2}f}{Q^{1/2}g}_{L^2(G)}.
\]
The space $L^2(0,T;L^2_Q)$ will be denoted simply by $H$.  We refer
to, e.g., \cite{nualart} for notation and terminology pertinent to
Malliavin calculus, as well as to \cite{Sanz-libro} for an exposition
geared towards SPDEs.

\smallskip

We shall write $a \lesssim b$ if there exists a constant $C$ such that
$a \leq Cb$. If the constant depends on some parameters of interest,
we indicate this as subscripts to the symbol. If $a \lesssim b$ and
$b \lesssim a$, we shall write $a \eqsim b$.
To shorten the notation of functional spaces, we set
$\bL^p:=L^p(\Omega)$, $L^p_x:=L^p(G)$, $L^p_xH:=L^p(G;H)$, and
$L^p_t:=L^p(0,T)$, for any $p$ for which they make sense. Moreover,
instead of writing, for instance, $L^p(\Omega;L^q(0,T))$ we shall sometimes
simply write $\bL^p L^q_t$, and similarly for other spaces.
We shall abbreviate deterministic and stochastic convolutions writing
\[
  S \ast g := \int_0^\cdot S(\cdot-s)g(s)\,ds \quad \text{and} \quad
  S \diamond G := \int_0^\cdot S(\cdot-s)G(s)\,dW(s).
\]

\subsection{Well-posedness}
The following well-posedness result for \eqref{eq:0}, that follows
from~\cite[Theorem~4.9]{KvN2} (cf.~\cite[Proposition~2.5]{cm:POTA21}),
allows to consider the Malliavin derivative of the mild solution
pointwise, i.e. for each $(t,x) \in [0,T] \times G$.
\begin{prop}
 \label{prop:0}
  Assume that
  \begin{equation}
    \label{eq:dqp}
  \frac{d}{2q} < \frac12 - \frac{1}{p}
  \end{equation}
  and $\sigma \colon \erre \to \erre$ is locally Lipschitz continuous
  with linear growth. If $u_0 \in L^p(\Omega;C(\overline{G}))$, then
  \eqref{eq:0} admits a unique $C(\overline{G})$-valued mild solution
  $u$, which satisfies the estimate
  \[
    \E \sup_{t \leq T} \norm[\big]{u(t)}_{C(\overline{G})}^p \lesssim_p
    1 + \E\norm[\big]{u_0}_{C(\overline{G})}^p.
  \]
\end{prop}
\noindent Condition~\eqref{eq:dqp} will be in force throughout the paper.

\subsection{Maximal estimates for convolutions}
We shall use two maximal estimates for deterministic and stochastic
convolutions. The first one is elementary, and the second one is a
special case of \cite[Proposition~4.2]{vNVW} (the proof of which is an
extension of the factorization method introduced in \cite{DPKZ}).

Let $S$ be a strongly continuous analytic semigroup on a Banach space
$E$ with generator $-A$.\footnote{The semigroup and its generator do
  not need to satisfy the assumptions of \S\ref{ssec:ass}.} We shall
denote by $E_\alpha$, $\alpha \in \erre$, the usual domains of powers
of (suitable shifts of) $A$, if $\alpha \geq 0$, and the corresponding
extrapolation spaces if $\alpha<0$. If $E=L^q$ or $L^q(G;H)$, with $H$
a Hilbert space, we shall write $E^q_\alpha$ or $E^q_\alpha(H)$,
respectively.

\begin{lemma}
  \label{lm:DPKZ}
  Let $r \in \mathopen]1,\infty\mathclose]$ and
  $\eta \in [0,1/r'\mathclose[ = [0,1-1/r\mathclose[$. Then there exists
  $\varepsilon \in \erre_+$ such that
  \[
    \norm[\big]{S \ast g}_{C([0,T];E_\eta)} \lesssim T^\varepsilon
    \norm[big]{g}_{L^r(0,T;E)}.
  \]
\end{lemma}
\begin{proof}
  By the analyticity of $S$,
  \begin{align*}
    \norm[\bigg]{\int_0^t S(t-s)g(s)\,ds}_{E_\eta}
    &\leq \int_0^t (t-s)^{-\eta} {\norm{g(s)}}_E\,ds\\
    &\leq \biggl( \int_0^t s^{-\eta r'}\,ds \biggr)^{1/r'} 
    \norm[big]{g}_{L^r(0,T;E)},
  \end{align*}
  where $\eta r' \in [0,1\mathclose[$.
\end{proof}

\begin{lemma}
  \label{lm:msc}
  Let $E$ be a UMD Banach space, and the positive constants
  $\alpha$, $p$, and $\eta$ be such that
  \[
    \alpha \in \mathopen]0,1/2\mathclose[, \quad 
    p \in \mathopen]2,\infty\mathclose[, \quad
    \theta < \alpha - \frac{1}{p}.
  \]
  Then, for any $G\colon \Omega \times [0,T] \to \cL(U;E)$ such that
  $Gu$ is measurable and adapted for every $u \in U$, there exists
  $\varepsilon \in \erre_+$ such that
  \[
  \norm[\big]{S \diamond G}_{L^p(\Omega;C([0,T];E_\eta)}
  \lesssim T^{\varepsilon} \norm[\big]{(t-\cdot)^{-\alpha}%
    G}_{L^p(\Omega;L^p(0,T;\gamma(L^2(0,t;U);E)))}.
  \]
\end{lemma}

\subsection{It\^o's formula}
Let $E$ be a UMD Banach space. If $\Phi \in \cL_2(E)$, i.e. $\Phi$ is
a continuous bilinear form on $E$, and $T \in \gamma(U;E)$, we set
\[
\operatorname{Tr}_T \Phi := \sum_{n\in\enne} \Phi(Th_n,Th_n),
\]
for which it is easily seen that
\begin{equation}
  \label{eq:tr-ineq}
  \abs{\operatorname{Tr}_T \Phi} \leq \norm{\Phi}_{\cL_2(E)}
  \norm{T}^2_{\gamma(U;E)}.
\end{equation}
We shall repeatedly use the following It\^o formula, proved
in~\cite{Brz:Ito}.
\begin{lemma}
  Consider the $E$-valued process
  \[
  u(t) = u_0 + \int_0^t b(s)\,ds + \int_0^t G(s)\,dW(s),
  \]
  where
  \begin{itemize}
  \item[(a)] $u_0 \colon \Omega \to E$ is $\cF_0$-measurable;
  \item[(b)] $b \colon \Omega \times [0,T] \to E$ is measurable, adapted and
    with paths in $L_1(0,T;E)$;
  \item[(c)] $G \colon \Omega \times [0,T] \to \cL(U;E)$ is
    $U$-measurable, adapted, stochastically integrable with respect to
    $W$, and with paths in $L_2(0,T;\gamma(U;E))$.
  \end{itemize}
  For any $\varphi \in C^2(E)$, one has
  \begin{align*}
    \varphi(u(t)) &= \varphi(u_0) + \int_0^t \varphi'(u(s))b(s)\,ds
    + \int_0^t \varphi'(u(s))G(s)\,dW(s)\\
    &\quad + \frac12 \int_0^t \operatorname{Tr}_{G(s)} \varphi''(u(s))\,ds.
  \end{align*}
\end{lemma}
\noindent The It\^o formula will be applied to functions of the type
$\norm{\cdot}^q_{L^q(G;H)}$, with $H$ a Hilbert space. The
differentiability of such functions is considered next.

\subsection{Differentiability of the $q$-th power of the norm in
  $L^q$ spaces of Hilbert-space-valued functions}
Let $H$ be a Hilbert space, $(X,\mathscr{A},\mu)$ a $\sigma$-finite
measure space, $q \in [2,\infty\mathclose[$, and denote the Bochner
space of $H$-valued functions $\phi$ such that $\norm{\phi}_H \in
L^q(\mu)$ by $L^q(\mu;H)$.
The duality map of $L^q(\mu;H)$, defined as the function $J \colon
L^q(\mu;H) \to L^{q'}(\mu;H)$ such that
\[
  \ip[\big]{\phi}{J(\phi)} = \norm[\big]{\phi}^2_{L^q(\mu;H)},
\]
where $\ip{\cdot}{\cdot}$ stands for the duality pairing between
$L^q(\mu;H)$ and $L^{q'}(\mu;H)$, is easily seen to be
\[
  J \colon \phi \longmapsto \norm[\big]{\phi}^{2-q}_{L^q(\mu;H)}
         \norm[\big]{\phi(\cdot)}^{q-2}_H \phi(\cdot).
\]
We shall need differentiability properties of a related functional,
namely of
\begin{align*}
  \Phi_q \colon L^q(\mu;H) &\longrightarrow \erre\\ 
  \phi &\longmapsto \norm[\big]{\phi}^q_{L^q(\mu;H)}.
\end{align*}
\begin{prop}
  One has $\Phi_q \in C^2(L^q(\mu;H))$, with
\begin{align*}
  \Phi_q'(u)\colon v
  &\longmapsto q \ip[\big]{\norm{u}_H^{q-2}u}{v} =
    q \int \norm{u(x)}_H^{q-2} {\ip{u(x)}{v(x)}}_H\,\mu(dx),\\
  \Phi_q''(u)\colon (v,w)
  &\longmapsto q \ip[\big]{\norm{u}_H^{q-2}v}{w}
    + q(q-2) \int \norm{u(x)}_H^{q-4} {\ip{u(x)}{v(x)}}_H
      {\ip{u(x)}{w(x)}}_H \,\mu(dx).
\end{align*}
\end{prop}
\noindent In particular, for any $u \in L^q(\mu;H)$,
\begin{equation}\label{eq:nJ}
  \Phi_q'(u) = q \norm[\big]{u}_{L^q(\mu;H)}^{q-2} J(u), \qquad
  \norm[\big]{\Phi'(u)}_{L^{q'}(\mu;H)} = q \norm[\big]{u}_{L^q(\mu;H)}^{q-1},
\end{equation}
and, by H\"older's inequality,
\begin{equation}
  \label{eq:n2}
  \norm[\big]{\Phi_q''(u)}_{\cL_2(L^q(\mu;H))}
  \leq q(q-1) \norm[\big]{u}^{q-2}_{L^q(\mu;H)}.
\end{equation}
The statement about the first-order derivative follows by general
properties of the duality mapping, while the statement about the
second-order derivative can be obtained either by a direct
computation, or showing that
$\norm{\cdot}_H^2 \colon L^q(\mu;H) \to L^{q/2}(\mu)$ is of class
$C^1$, and then applying calculus rules for the composition of
differentiable functions between Banach spaces (see,
e.g.,~\cite{AmbPro}).


\section{The formal equation for Malliavin derivatives}
Let us write the mild formulation of \eqref{eq:0} in the equivalent
form
\begin{align*}
  u(t,x) &= u_0(x) + \int_0^t\!\!\int_G K_{t-s}(x,y) f(u(s,y))\,dy\,ds\\
         &\quad + \int_0^t\!\!\int_G K_{t-s}(x,y) \sigma(u(s,y))
           \,\bar{W}(dy,ds),
\end{align*}
where $\bar{W}$ is a Brownian sheet on $[0,T] \times G$ with
covariance (in space) equal to $Q$ (see~\cite{cm:POTA21} and
references therein).

If $Du(t,x)$ exists for all $(t,x) \in [0,T] \times G$, taking the
Malliavin derivative of both sides yields
\begin{align*}
  Du(t,x) 
  &= v_0(t,x) 
    + \int_0^t\!\!\int_G K_{t-s}(x,y) f'(u(s,y)) 
    Du(s,y)\,dy\\
  &\quad + \int_0^t\!\!\int_G K_{t-s}(x,y) \sigma'(u(s,y)) 
    Du(s,y)\,\bar{W}(dy,ds),
\end{align*}
where 
$v_0\colon \Omega \times [0,T] \times G \to H$ is defined by
\[
  v_0(t,x) := (\tau,z) \mapsto K_{t-\tau}(x,z) \,
  \sigma(u(\tau,z)) \, 1_{[0,t]}(\tau).
\]
It was shown in \cite{cm:POTA21} that $u(t,x) \in \bD^{1,p}_\loc$ for
every $(t,x) \in [0,T] \times G$, and that, for every $x \in G$,
$Du(\cdot,x) = Du_n(\cdot,x)$ on $\cc{0}{T_n}$ for every
$n \in \mathbb{N}$, where is defined $T_n$ is the first time $t$ when
$\norm{u(t)}_{C(\overline{G})}$ reaches $n$, and $u_n$ is the unique
mild solution in $L^p(\Omega;C([0,T];C(\overline{G})))$ to the
equation
\[
  du_n + Au_n\,dt = f_n(u_n)\,dt + \sigma(u_n) B\,dW,
  \qquad u_n(0)=u_0,
\]
with
$f_n\colon x \mapsto f(x)1_{[-n,n]}(x) +
f(nx/\abs{x})1_{[-n,n]^\complement}$. Moreover, $u_n=u$ on
$\cc{0}{T_n}$ for every $n \in \mathbb{N}$. In particular, since $f_n$
is Lipschitz continuous, it follows by \cite[Theorem~3.1]{cm:POTA21}
that $Du_n \in L^\infty(0,T;L^p(\Omega;L^\infty(G;H)))$, hence, by
construction of $Du$, one has at least
$Du \in L^0(\Omega \times [0,T];L^\infty(G;H))$.

Looking at the above identity for $Du$ as an equation for
$L^q(G;H)$-valued processes, writing $v(t,x) = Du(t,x)$, one is lead
to considering the equation
\begin{equation}
\label{eq:v}
  v(t) = v_0(t)
  + \int_0^t S(t-s) f'(u(s)) v(s)\,ds
  + \int_0^t S(t-s) \sigma'(u(s)) v(s) B\,dW(s),
\end{equation}
where, for any $\eta \in \mathopen] d/(2q),1/2-1/p \mathclose[$,
\[
  \norm[\big]{v_0(t)}_{L^\infty(G;H)} \lesssim
  \norm[\big]{S(t-\cdot)\sigma(u_\lambda)B}_{L^2(0,t;\gamma(U;E^q_\eta))},
\]
hence
\[
  \norm[\big]{v_0}_{L^\infty([0,T] \times G;H)} \lesssim
  \bigl( 1 + \norm[\big]{\sigma(u)}_{C([0,T];C(\overline{G}))} \bigr)
  \norm[\big]{B}_{\gamma(U;L^q)} T^{(1-2\eta)/2}.
\]
This implies
\begin{align*}
  \norm[\big]{v_0}_{L^p(\Omega;L^\infty([0,T] \times G;H))}
  &\lesssim_T \Bigl(1 + \norm[\big]{u}_{L^p(\Omega;C([0,T];C(\overline{G})))}
    \Bigr) \norm[\big]{B}_{\gamma(U;L^q)}\\
  &\lesssim \Bigl(1 + \norm[\big]{u_0}_{C(\overline{G})} \Bigr)
    \norm[\big]{B}_{\gamma(U;L^q)}.
\end{align*}
A mild solution to \eqref{eq:v} can be constructed by localization
arguments. As a first step, let us replace $f_n$, without loss of
generality, with the smoother version defined by
\[
f_n(x) := f(0) + \int_0^x f'(y) \chi_n(y)\,dy,
\]
with $\chi_n\colon \erre \to [0,1]$ of class $C^1$ such that
$\abs{\chi'_n} \leq 1$ and
\[
\operatorname{supp} \chi_\lambda = [-n-1,n+1],
\qquad \bigl.\chi_n\bigr\vert_{[-n,n]} = 1.
\]
Let us define the process $F_n := f'_n(u)$, and consider the equation
\begin{equation}
\label{eq:vn}
  v_n(t) = v_0(t)
  + \int_0^t S(t-s) F_n(s) v_n(s)\,ds
  + \int_0^t S(t-s) \sigma'(u(s)) v_n(s) B\,dW(s),
\end{equation}
that, thanks to the boundedness of $F_n$ and $\sigma'$, admits a
(global) mild solution $v_n$ belonging to $\bL^p C_t L^\infty_xH$,
which is unique also in the larger space $L^\infty_t \bL^p
L^q_xH$. Defining $v$ to be the process equal to $v_n$ on
$\cc{0}{T_n}$ for every $n \in \mathbb{N}$, recalling that $T_n$
converges monotonically to $+\infty$, we obtain a mild solution to
\eqref{eq:v} that is necessarily unique in the set of processes that
are locally in $L^\infty_t \bL^p L^q_xH$. Therefore $v=Du$ in
$L^0(\Omega \times [0,T];L^\infty(G;H))$ and, outside a set $N \subset
\Omega \times [0,T]$ of $\P\otimes dt$-measure zero, $Du
1_{\cc{0}{T_n}} \in \bL^p C_t L^\infty_xH$.

Even though $v=Du$ is locally (i.e. on increasing stochastic
intervals) in $\bL^p C_t L^\infty_xH$, it seems not possible to obtain
uniform bounds in this space using only the mild form of the
equation. Without uniform bounds, even in weaker norms, it is not
clear how to show that $u(t,x)$ belongs to any space $\bD^{1,p}$.  The
main obstacle to obtaining such uniform bounds is the deterministic
convolution term in \eqref{eq:v}, essentially because any estimate of
the $L^p(\Omega)$ norm of $f'(u)v$, or of $f'_n(u)v$, will necessarily
involve the $L^s(\Omega)$ norm of $v$, for some $s>p$, as $f'(u)$ is
not in $L^\infty(\Omega)$, or, analogously, the norm of $f'_n(u)$ in
$L^\infty(\Omega)$ may explode as $n$ tends to infinity. Note,
however, that such (admittedly crude) estimates do not take advantage
in any way of the dissipativity of $f$. Estimates that do exploit the
dissipative character of the equation will be obtained in the next
section, at the cost, however, of a kind of smoothness assumption on
the noise.


\section{Estimates with smooth noise}
Throughout this section we assume that $B$ is very regular, i.e. that
$B \in \gamma(U;X)$, with $X$ a Banach space continuously embedded in
$L^\infty(G)$, which could be, for instance, $E^q_\alpha$ with
$\alpha > d/(2q)$. Note that if $B_0 \in \gamma(U;L^q)$, the ideal
property of $\gamma$-Radonifying operators implies that
$B_\varepsilon := (I+\varepsilon A)^{-\alpha}B_0$ belongs to
$\gamma(U;E^q_\alpha)$ for every $\varepsilon>0$ and
$\lim_{\varepsilon \to 0} B_\varepsilon = B$ in $\gamma(U;L^q)$.

Let $(f_\lambda)_{\lambda>0} \subset C^1(\erre)$ be a collection of
decreasing Lipschitz-continuous functions such that $f_\lambda$ and
$f'_\lambda$ converge pointwise as $\lambda \to 0$ to $f$ and $f'$,
respectively. For instance, one may take (as in the previous section)
\[
f_\lambda(x) := f(0) + \int_0^x f'(y)\chi_\lambda(y)\,dy,
\]
with $\chi_\lambda\colon \erre \to [0,1]$ of class $C^1$ such that
$\abs{\chi'_\lambda} \leq 1$ and
\[
\operatorname{supp} \chi_\lambda = [-1/\lambda-1,1/\lambda+1],
\qquad \bigl.\chi_\lambda\bigr\vert_{[-1/\lambda,1/\lambda]} = 1,
\]
or the Yosida approximation
\[
  f_\lambda = \frac{1}{\lambda}\bigl(I - (I-\lambda f)^{-1} \bigr),
\]
where $I\colon \erre \to \erre$ is the identity function.

Recall that the equation
\[
  du_\lambda + Au_\lambda\,dt = f_\lambda(u_\lambda)\,dt +
  \sigma(u_\lambda) B\,dW, \qquad u_\lambda(0)=u_0.
\]
admits a unique mild solution
$u_\lambda \in L^p(\Omega;C([0,T];C(\overline{G})))$ for every $p>0$,
because $u_0 \in C(\overline{G})$ is non-random. Therefore, as proved
in \cite{cm:POTA21}, $u_\lambda(t,x) \in \bD^{1,p}$ for every
$(t,x) \in [0,T] \times G$ and every $p>0$, with
$Du_\lambda \in L^\infty(0,T;L^p(\Omega;L^\infty(G;H)))$, and
\begin{align*}
  Du_\lambda(t,x) 
  &= v_{0,\lambda}(t) 
    + \int_0^t\!\!\int_G K_{t-s}(x,y) f'_\lambda(u_\lambda(s,y)) 
    Du_\lambda(s,y)\,dy\\
  &\quad + \int_0^t\!\!\int_G K_{t-s}(x,y) \sigma'_\lambda(u_\lambda(s,y)) 
    Du_\lambda(s,y)\,\bar{W}(dy,ds),
\end{align*}
where 
\[
  v_{0,\lambda}(t,x) := (\tau,z) \mapsto K_{t-\tau}(x,z) \,
  \sigma(u_\lambda(\tau,z)) \, 1_{[0,t]}(\tau).
\]
We interpret this as an equation for the $L^q(G;H)$-valued process
$v_\lambda$, with $v_\lambda(t)\colon x \mapsto Du_\lambda(t,x)$,
namely
\begin{equation}
\label{eq:vl}
  v_\lambda(t) = v_{0,\lambda}(t)
  + \int_0^t S(t-s) f'_\lambda(u_\lambda(s)) v_\lambda(s)\,ds
  + \int_0^t S(t-s) \sigma'(u_\lambda(s)) v_\lambda(s) B\,dW(s).
\end{equation}
In complete similarity to the previous section one has, for any
$\eta \in \mathopen] d/(2q),1/2-1/p \mathclose[$,
\[
  \norm[\big]{v_{0,\lambda}(t)}_{L^\infty(G;H)} \lesssim
  \norm[\big]{S(t-\cdot)\sigma(u_\lambda)B}_{L^2(0,t;\gamma(U;E^q_\eta))},
\]
hence
\[
    \norm[\big]{v_{0,\lambda}}_{\bL^p L^\infty_{t,x}H}
    \lesssim_p \bigl(1 + \norm{u_0}_{C(\overline{G})} \bigr)
      \norm[\big]{B}_{\gamma(U;L^q)} \, T^{(1-2\eta)/2}.
\]
Moreover, since $f'_\lambda$ is bounded, \eqref{eq:vl} admits a unique
solution $v_\lambda \in \bL^p C_t L^\infty_xH$.

\begin{prop}
  \label{prop:vl}
  The family of processes $(v_\lambda)_{\lambda>0}$ is bounded in
  $L^p(\Omega;C([0,T];L^q(G;H)))$.
\end{prop}
\begin{proof}
  Setting $\widetilde{v}_\lambda := v_\lambda - v_{0,\lambda}$, in
  view of the boundedness of $v_{0,\lambda}$ it is enough to show
  that $(\widetilde{v}_\lambda)$ is bounded in
  $L^p(\Omega;L^\infty(0,T;L^q(G;H)))$. To this purpose, let us write
  \begin{align*}
    \widetilde{v}_\lambda(t)
    &= \int_0^t S(t-s) f'_\lambda(u_\lambda(s)) \bigl(
       \widetilde{v}_\lambda(s) + v_{0,\lambda}(s) \bigr)\,ds\\
    &\quad \hspace{3em} + \int_0^t S(t-s) \sigma'(u_\lambda(s)) \bigl(
       \widetilde{v}_\lambda(s) + v_{0,\lambda}(s) \bigr)B\,dW(s),
  \end{align*}
  that is the mild form of the differential equation
  \[
    d\widetilde{v}_\lambda + A\widetilde{v}_\lambda\,dt =
    f'_\lambda(u_\lambda) \bigl( \widetilde{v}_\lambda + v_{0,\lambda}
    \bigr)\,dt + \sigma'(u_\lambda) \bigl( \widetilde{v}_\lambda +
    v_{0,\lambda} \bigr)B\,dW
  \]
  with initial condition $\widetilde{v}_\lambda(0)=0$.  It\^o's
  formula for the $q$-th power of the $L^q(G;H)$ norm applied to a
  suitable semimartingale approximation of $\widetilde{v}_\lambda$
  (see, e.g., \cite{cm:SIMA18} for details) yields
  \begin{equation}
    \label{eq:Itoq}
    \begin{split}
      &\norm[\big]{\widetilde{v}_\lambda(t)}_{L^q_xH}^q + q \int_0^t
      \norm[\big]{\widetilde{v}_\lambda(t)}_{L^q_xH}^{q-2}
      \ip[\big]{A\widetilde{v}_\lambda}{J(\widetilde{v}_\lambda)}\,ds\\
      &\hspace{3em} \leq q \int_0^t
      \norm[\big]{\widetilde{v}_\lambda}_{L^q_xH}^{q-2}
      \ip[\big]{f'_\lambda(u_\lambda)\bigl( \widetilde{v}_\lambda
        + v_{0,\lambda} \bigr)}{J(\widetilde{v}_\lambda)}\,ds\\
      &\hspace{3em}\quad + q \int_0^t
      \norm[\big]{\widetilde{v}_\lambda}_{L^q_xH}^{q-2}
      J(\widetilde{v}_\lambda) \sigma'(u_\lambda)
      \bigl( \widetilde{v}_\lambda + v_{0,\lambda} \bigr) B\,dW(s)\\
      &\hspace{3em}\quad + \frac12 q(q-1) \int_0^t
      \norm[\big]{\sigma'(u_\lambda)%
        \bigl( \widetilde{v}_\lambda + v_{0,\lambda} \bigr)
        B}^2_{\gamma(U;L^q_xH)}
      \norm[\big]{\widetilde{v}_\lambda}^{q-2}_{L^q_xH}\,ds.
    \end{split}
  \end{equation}
  Recalling that $f_\lambda$ is decreasing, one has
  \begin{align*}
    \ip[\big]{f'_\lambda(u_\lambda)\bigl( \widetilde{v}_\lambda
    + v_{0,\lambda} \bigr)}{J(\widetilde{v}_\lambda)}
    &\leq
      \ip[\big]{f'_\lambda(u_\lambda)v_{0,\lambda}}{J(\widetilde{v}_\lambda)}\\
    &\leq \frac12 \norm[\big]{\widetilde{v}_\lambda}^2_{L^q_xH}
      + \frac12 \norm[\big]{f'_\lambda(u_\lambda)v_{0,\lambda}}^2_{L^q_xH}\\
    &\leq \frac12 \norm[\big]{\widetilde{v}_\lambda}^2_{L^q_xH}
      + \frac12 \norm[\big]{f'_\lambda(u_\lambda)}^2_{L_x^q}
      \norm[\big]{v_{0,\lambda}}^2_{L^\infty_xH},
  \end{align*}
  hence
  \begin{align*}
    &\norm[\big]{\widetilde{v}_\lambda}_{L^q_xH}^{q-2}
      \ip[\big]{f'_\lambda(u_\lambda)\bigl( \widetilde{v}_\lambda
      + v_{0,\lambda} \bigr)}{J(\widetilde{v}_\lambda)}\\
    &\hspace{3em} \leq \frac12 \norm[\big]{\widetilde{v}_\lambda}_{L^q_xH}^q
      + \frac12 \norm[\big]{\widetilde{v}_\lambda}_{L^q_xH}^{q-2}
      \norm[\big]{f'_\lambda(u_\lambda)}^2_{L_x^q}
      \norm[\big]{v_{0,\lambda}}^2_{L^\infty_xH},
  \end{align*}
  where, by Young's inequality with conjugate exponents $q/(q-2)$ and
  $q/2$,
  \begin{align*}
    &\norm[\big]{\widetilde{v}_\lambda}_{L^q_xH}^{q-2}
      \norm[\big]{f'_\lambda(u_\lambda)}^2_{L_x^q}
      \norm[\big]{v_{0,\lambda}}^2_{L^\infty_xH}\\
    &\hspace{3em} \leq
      \frac{q-2}{q} \norm[\big]{\widetilde{v}_\lambda}_{L^q_xH}^q
      + \frac{2}{q} \norm[\big]{f'_\lambda(u_\lambda)}^q_{L_x^q}
      \norm[\big]{v_{0,\lambda}}^q_{L^\infty_xH}.
  \end{align*}
  The first term on the right-hand side of \eqref{eq:Itoq} is thus
  estimated by
  \begin{align*}
    &\frac{q-2}{2} \int_0^t \norm[\big]{\widetilde{v}_\lambda}_{L^q_xH}^q\,ds
      + \int_0^t \norm[\big]{f'_\lambda(u_\lambda)}^q_{L_x^q}
      \norm[\big]{v_{0,\lambda}}^q_{L^\infty_xH}\,ds\\
    &\hspace{3em} = \frac{q-2}{2} \norm[\big]{%
      \widetilde{v}_\lambda}_{L^q(0,t;L^q_xH)}^q
      + \norm[\Big]{\norm[\big]{f'_\lambda(u_\lambda)}_{L_x^q}
      \norm[\big]{v_{0,\lambda}}_{L^\infty_xH}}^q_{L^q(0,t)}.
  \end{align*}
  Analogously, using the ideal property of $\gamma$-Radonifying
  operators on the diagram
  \[
    U \xrightarrow{B} X \embed L_x^\infty
    \xrightarrow{\sigma'(u)} L^\infty_x
    \xrightarrow{\widetilde{v}_\lambda+v_{0,\lambda}} L^q_xH,
  \]
  denoting the Lipschitz constant of $\sigma$ by $L_\sigma$, one has
  \begin{align*}
    &\norm[\big]{\sigma'(u_\lambda)%
      \bigl( \widetilde{v}_\lambda + v_{0,\lambda} \bigr) B}^2_{\gamma(U;L^q_xH)}
      \norm[\big]{\widetilde{v}_\lambda}^{q-2}_{L^q_xH}\\
    &\hspace{3em} \leq L_\sigma^2 \norm[\big]{\widetilde{v}_\lambda %
      + v_{0,\lambda}}^2_{L^q_xH} \norm[\big]{B}^2_{\gamma(U;X)}
      \norm[\big]{\widetilde{v}_\lambda}^{q-2}_{L^q_xH}\\
    &\hspace{3em} \leq L_\sigma^2 \norm[\big]{B}^2_{\gamma(U;X)} \biggl(
      \frac{q-2}{q} \norm[\big]{\widetilde{v}_\lambda}^q_{L^q_xH}
      + \frac{2}{q} \norm[\big]{\widetilde{v}_\lambda %
      + v_{0,\lambda}}^q_{L^q_xH} \biggr).
  \end{align*}
  The third term on the right-hand side of \eqref{eq:Itoq} is hence
  estimated by
  \begin{align*}
    &L_\sigma^2 \norm[\big]{B}^2_{\gamma(U;X)} \biggl( \frac12 (q-1)(q-2)
      \int_0^t \norm[\big]{\widetilde{v}_\lambda}^q_{L^q_xH}\,ds
      + (q-1) \int_0^t \norm[\big]{\widetilde{v}_\lambda
      + v_{0,\lambda}}^q_{L^q_xH}\,ds \biggr)\\
    &\hspace{3em} = L_\sigma^2 \norm[\big]{B}^2_{\gamma(U;X)} \biggl(
      \frac12 (q-1)(q-2) \norm[\big]{\widetilde{v}_\lambda}^q_{L^q(0,t;L^q_xH)}
      + (q-1) \norm[\big]{\widetilde{v}_\lambda+v_{0,\lambda}}^q_{L^q(0,t;L^q_xH}
      \biggr).
  \end{align*}
  Let us denote the stochastic integral on the right-hand side of
  \eqref{eq:Itoq}, which is a real local martingale, by $M$, and set
  \begin{align*}
    C_1 = C_1(q,B) &:= \Bigl( \frac12(q-2) + \frac12 q(q-1) L_\sigma^2
                     \norm[\big]{B}^2_{\gamma(U;X)} \Bigr)^{1/q},\\
    C_2 = C_2(q,B) &:= (q-1)^{1/q} \bigl( L_\sigma
                     {\norm{B}}_{\gamma(U;E_\eta)} \bigr)^{2/q}.
  \end{align*}
  Then
  \begin{align*}
    \norm[\big]{\widetilde{v}_\lambda}_{C_t L^q_xH} 
    &\leq C_1 \norm[\big]{\widetilde{v}_\lambda(t)}_{L^q_t L^q_xH}
      + C_2 \norm[\big]{v_{0,\lambda}(t)}_{L^q_t L^q_xH}\\
    &\quad + \norm[\Big]{\norm[\big]{f'_\lambda(u_\lambda)}_{L_x^q}
      \norm[\big]{v_{0,\lambda}}_{L^\infty_xH}}_{L^q_t}
      + q^{1/q} \bigl(M_T^*\bigr)^{1/q},
  \end{align*}
  thus also
  \begin{align*}
    \norm[\big]{\widetilde{v}_\lambda}_{\bL^p C_t L^q_xH}
    &\leq C_1 \norm[\big]{\widetilde{v}_\lambda(t)}_{\bL^p L^q_t L^q_xH}
      + C_2 \norm[\big]{v_{0,\lambda}(t)}_{\bL^p L^q_t L^q_xH}\\
    &\quad + \norm[\Big]{\norm[\big]{f'_\lambda(u_\lambda)}_{L_x^q}
      \norm[\big]{v_{0,\lambda}}_{L^\infty_xH}}_{\bL^p L^q_t} 
      + q^{1/q} \norm[\big]{M_T^*}^{1/q}_{\bL^{p/q}}.
  \end{align*}
  The Burkholder-Davis-Gundy inequality yields
  \[
    \norm[\big]{M_T^*}^{1/q}_{\bL^{p/q}} \eqsim
    \norm[\big]{[M,M]_T^{1/2}}^{1/q}_{\bL^{p/q}} =
    \norm[\big]{[M,M]_T^{1/(2q)}}_{\bL^p},
  \]
  where
  \begin{align*}
    [M,M]^{1/2}_T 
    &\leq  L_\sigma \norm{B}_{\gamma(U;X)} \biggl(\int_0^T 
      \norm[\big]{\widetilde{v}_\lambda}^{2(q-1)}_{L^q_xH}
      \norm[\big]{\widetilde{v}_\lambda + v_{0,\lambda}}^2_{L^q_xH} \,ds
      \biggr)^{1/2}\\
    &= L_\sigma \norm{B}_{\gamma(U;X)} \norm[\Big]{%
      \norm[\big]{\widetilde{v}_\lambda}^{q-1}_{L^q_xH}%
      \norm[\big]{\widetilde{v}_\lambda + v_{0,\lambda}}_{L^q_xH}}_{L^2_t}
  \end{align*}
  and, all norms being on $L^q(G;H)$,
  \begin{align*}
    \norm{\widetilde{v}_\lambda}^{q-1}
    \norm{\widetilde{v}_\lambda + v_{0,\lambda}}
    &\leq \norm{\widetilde{v}_\lambda}^q +
      \norm{\widetilde{v}_\lambda}^{q-1} \norm{v_{0,\lambda}}\\
    &\leq \norm{\widetilde{v}_\lambda}^q 
      + \frac{q-1}{q}\norm{\widetilde{v}_\lambda}^q
      + \frac{1}{q} \norm{v_{0,\lambda}}^q,
  \end{align*}
  so that
  \[
    [M,M]_T^{\frac12 \cdot \frac1q} \leq L_\sigma^{1/q}
    \norm{B}^{1/q}_{\gamma(U;X)} \Bigl( (2-1/q)^{1/q}
    \norm[\big]{\widetilde{v}_\lambda}_{L^{2q}_t L^q_xH} + (1/q)^{1/q}
    \norm[\big]{v_{0,\lambda}}_{L^{2q}_t L^q_xH} \Bigr)
  \]
  and
  \[
    \norm[\big]{[M,M]_T^{1/(2q)}}_{\bL^p} \leq L_\sigma^{1/q}
    \norm{B}^{1/q}_{\gamma(U;X)} \Bigl( (2-1/q)^{1/q}
    \norm[\big]{\widetilde{v}_\lambda}_{\bL^p L^{2q}_t L^q_xH}
    + (1/q)^{1/q} \norm[\big]{v_{0,\lambda}}_{\bL^p L^{2q}_t L^q_xH} \Bigr).
  \]
  Setting
  \begin{align*}
    C_3 = C_3(q,B) &:= (2q-1)^{1/q} L_\sigma^{1/q} \norm{B}^{1/q}_{\gamma(U;X)},\\
    C_4 = C_4(q,B) &:= L_\sigma^{1/q} \norm{B}^{1/q}_{\gamma(U;X)}
  \end{align*}
  we are left with
  \begin{align*}
    \norm[\big]{\widetilde{v}_\lambda}_{\bL^p C_t L^q_xH}
    &\leq C_1 \norm[\big]{\widetilde{v}_\lambda}_{\bL^p L^q_t L^q_xH}
      + C_2 \norm[\big]{v_{0,\lambda}}_{\bL^p L^q_t L^q_xH}\\
    &\quad + C_3 \norm[\big]{\widetilde{v}_\lambda}_{\bL^p L^{2q}_t L^q_xH}
      + C_4 \norm[\big]{v_{0,\lambda}}_{\bL^p L^{2q}_t L^q_xH}\\
    &\quad + \norm[\Big]{\norm[\big]{f'_\lambda(u_\lambda)}_{L_x^q}
      \norm[\big]{v_{0,\lambda}}_{L^\infty_xH}}_{\bL^p L^q_t},
  \end{align*}
  where
  \[
    \norm[\Big]{\norm[\big]{f'_\lambda(u_\lambda)}_{L_x^q}
      \norm[\big]{v_{0,\lambda}}_{L^\infty_xH}}_{\bL^p L^q_t}
   \leq \norm[\big]{f'_\lambda(u_\lambda)}_{\bL^{2p} L^q_{t,x}}
      \norm[\big]{v_{0,\lambda}}_{\bL^{2p} L^\infty_{t,x}H},
  \]
  and recalling that $f'_\lambda(x) \lesssim 1+\abs{x}^m$ uniformly
  with respect to $\lambda$ (see, e.g., \cite[p.~295]{cm:POTA13}) and
  $u_\lambda \to u$ in $L^p(\Omega;C([0,T];C(\overline{G})))$ for
  every $p>2$,
  \[
    \norm[\big]{f'_\lambda(u_\lambda)}_{\bL^{2p} L^q_{t,x}}
    \lesssim 1 + \norm[\big]{u_\lambda}^m_{\bL^{2mp} L^{mq}_{t,x}}
    \lesssim 1 + \norm[\big]{u}^m_{\bL^{2mp} C_{t,x}}
    \lesssim 1 + \norm[\big]{u_0}^m_{C_x},
  \]
  hence
  \[
    \norm[\Big]{\norm[\big]{f'_\lambda(u_\lambda)}_{L_x^q}
      \norm[\big]{v_{0,\lambda}}_{L^\infty_xH}}_{\bL^p L^q_t}
    \lesssim 1 + \norm[\big]{u_0}^{m+1}_{C(\overline{G})},
  \]
  with an implicit constant that depends, among others, on $p$.
  Since the norm of the continuous embedding
  $L^\infty(0,T) \embed L^r(0,T)$ is $T^{1/r}$ for any $r \geq 1$, we
  conclude that, for $T_0$ sufficiently small,
  $(\widetilde{v}_\lambda)_{\lambda>0}$ is bounded in
  $L^p(\Omega;L^\infty(0,T_0;L^q(G;H)))$. By an iteration argument,
  the same statement is true with $T_0$ replaced by $T$.
\end{proof}

\begin{rmk}
  Monotonicity techniques would provide sharper results if it were
  possible to write either $v_{0,\lambda}(t) = S(t) \zeta_{0,\lambda}$
  for some $\zeta_{0,\lambda} \in L^q(G;H)$ or
  $v_{0,\lambda} = S \ast \zeta_\lambda$, for a process
  $\zeta_\lambda$ with paths in $L^1(0,T;L^q(G;H))$. Unfortunately
  $v_{0,\lambda}(t)$ seems to be too ``singular'' to allow such
  representations. For this reason we proceeded by changing variable
  in the previous proof.
\end{rmk}

The boundedness of $(v_\lambda)$ just established immediately implies
the following compactness properties in suitable weak topologies.
\begin{coroll}
  There exist $\zeta \in L^\infty(0,T;L^p(\Omega;L^q(G;H)))$ such that
  \begin{align*}
  v_\lambda &\longrightarrow \zeta \quad \text{weakly* in }
  L^\infty(0,T;L^p(\Omega;L^q(G;H))),\\
  v_\lambda &\longrightarrow \zeta \quad \text{weakly in }
  L^p(\Omega \times [0,T]; L^q(G;H)).
  \end{align*}
\end{coroll}

Passing to the limit as $\lambda \to 0$ in \eqref{eq:vl}, recalling
that the linear operators $\phi \mapsto S \ast \phi$ and
$\Phi \mapsto S \diamond \Phi$ are continuous, hence also continuous
with respect to the corresponding weak topologies, it follows that
that $\zeta$ coincides with the (unique) solution $v$ to the equation
for formal derivatives obtained in the previous section. This also
shows that, under the smoothness assumption on $B$, $Du$ belongs to
$L^\infty(0,T;L^p(\Omega;L^q(G;H)))$, after modification on a subset
of $\Omega \times [0,T]$ of measure zero. We shall see later that
better regularity of $Du$ can be obtained.

\medskip

We are now going to show how the above compactness results imply
estimates on the first-order Malliavin derivative of $u$. We recall
that in \cite{cm:POTA21} the basic result $u(t,x) \in \bD^{1,p}_\loc$ for
every $p \geq 1$ and $(t,x) \in [0,T] \times G$ was proved.
\begin{thm}
  Let $r = p \wedge q$. Then $u(t,x) \in \bD^{1,r}$ for almost every
  $(t,x) \in [0,T] \times G$.
\end{thm}
\begin{proof}
  Since $v_\lambda$ converges to $v$ weakly in $L^p(\Omega \times
  [0,T];L^q(G;H)$ as $\lambda \to 0$, Mazur's lemma implies that there
  exists a sequence $(z_n)$ defined by
  \[
    z_n := \sum_{k=n}^{N(n)} \alpha_{n,k} v_{\lambda_k},
    \qquad \text{with }
    \alpha_{n,k} \in \erre_+, \; \sum_{k=1}^{N(n)} \alpha_{n,k}=1,
  \]
  such that $z_n \to v$ strongly in $L^p(\Omega \times
  [0,T];L^q(G;H)$. Let $(\widetilde{u}_n)$ be the sequence defined by
  \[
  \widetilde{u}_n := \sum_{k=n}^{N(n)} \alpha_{n,k} u_{\lambda_k}.
  \]
  Then $\widetilde{u}_n$ converges to $u$ in
  $L^p(\Omega;C(0,T];C(\overline{G}))$ as $n \to \infty$, and
  $z_n(t,x) = D\widetilde{u}_n(t,x)$ for every $(t,x) \in [0,T] \times
  G$ by linearity of $D$. Let $r:=p \wedge q$. Then
  $\widetilde{u}_n(t,x)$ converges to $u(t,x)$ in $L^r(\Omega)$ for
  every $(t,x) \in [0,T] \times G$ and $z_n$ converges to $v$ in
  $L^r([0,T] \times G;L^r(\Omega;H))$, i.e.
  \[
  \lim_{n \to \infty} \int_0^T\!\!\int_G \E\norm[\big]{z_n(t,x)
    - v(t,x)}^r_H \,dx\,dt = 0,
  \]
  hence, passing to a subsequence is necessary,
  \[
  \lim_{n \to \infty} z_n(t,x) = v(t,x) \qquad \text{in } L^r(\Omega;H)
  \]
  for almost all $(t,x) \in G_T$. By the closability of $D$ it follows
  that $u(t,x) \in \mathbb{D}^{1,r}$ and $Du(t,x)=v(t,x)$ for almost
  all $(t,x) \in [0,T] \times G$.
\end{proof}

In fact both $p$ and $q$ can be taken as large as needed, hence we
actually have that, for almost every $(t,x) \in [0,T] \times G$,
$u(t,x) \in \bD^{1,r}$ for every $r \geq 0$.

\medskip

As a further step, we are going to show that $v$ is the limit in a
stronger topology of solutions to approximating equations. To this
purpose, however, we are not going to use \eqref{eq:vl}, but another
approximation of \eqref{eq:v}.
Let us set, for every $t \in [0,T]$, recalling that $F \leq 0$,
\[
  F_\lambda(t) := \frac{f'(u(t))}{1-\lambda f'(u(t))}
  = \frac{F(t)}{1 - \lambda F(t)}, \qquad \lambda>0,
\]
and consider the equation
\begin{equation}
  \label{eq:vbl}
  \overline{v}_\lambda(t) = v_0(t) 
  + \int_0^t S(t-s) F_\lambda(s) \overline{v}_\lambda(s)\,ds
  + \int_0^t S(t-s) \sigma'(u(s)) \overline{v}_\lambda(s) B\,dW(s),
\end{equation}
which is readily seen to admit a unique solution
$\overline{v}_\lambda$ in $L^p(\Omega;C([0,T];L^\infty(G;H)))$, as it
follows by boundedness of $F_\lambda$.

\begin{prop}
  The family of processes $(\overline{v}_\lambda)_{\lambda>0}$ is
  a Cauchy net in $L^p(\Omega;C([0,T];L^q(\overline{G})))$.
\end{prop}
\begin{proof}
  It clearly holds
  \begin{align*}
    \overline{v}_\lambda(t) - \overline{v}_\mu(t) &= \int_0^t S(t-s)
    \bigl( F_\lambda(s) \overline{v}_\lambda(s)
    - F_\mu(s) \overline{v}_\mu(s) \bigr)\,ds\\
    &\quad + \int_0^t S(t-s) \sigma'(u(s)) \bigl(
    \overline{v}_\lambda(s) - \overline{v}_\mu(s) \bigr) B\,dW(s),
  \end{align*}
  i.e. $\overline{v}_\lambda - \overline{v}_\mu$ is the unique mild
  solution to the differential equation
  \[
  d(\overline{v}_\lambda - \overline{v}_\mu) 
  + A(\overline{v}_\lambda - \overline{v}_\mu)\,dt
  = (F_\lambda \overline{v}_\lambda - F_\mu \overline{v}_\mu)\,dt
  + \sigma'(u) ( \overline{v}_\lambda - \overline{v}_\mu ) B\,dW
  \]
  with zero initial condition. We are going to obtain estimates on the
  difference $\overline{v}_\lambda - \overline{v}_\mu$ applying
  It\^o's formula, which is formal but harmless, as already mentioned.
  In the following, $\ip{\cdot}{\cdot}$, without subscripts, stands
  for duality pairing $\ip{\cdot}{\cdot}$ induced by the scalar
  product of $L^2(G;H)$. One has
  \begin{equation}
  \label{eq:vblm}
  \begin{split}
    &\norm[\big]{\overline{v}_\lambda(t) - \overline{v}_\mu(t)}_{L^q_xH}^q
    + \int_0^t \ip[\big]{A(\overline{v}_\lambda - \overline{v}_\mu)}%
    {\Phi'_q(\overline{v}_\lambda - \overline{v}_\mu)}\,ds\\
    &\hspace{3em} \leq \int_0^t \ip[\big]{F_\lambda\overline{v}_\lambda%
      - F_\mu\overline{v}_\mu}{\Phi'_q(\overline{v}_\lambda%
      - \overline{v}_\mu)}\,ds\\
    &\hspace{3em}\quad + \int_0^t \Phi'_q(\overline{v}_\lambda%
      - \overline{v}_\mu) \sigma'(u) (\overline{v}_\lambda%
      - \overline{v}_\mu) B\,dW(s)\\
    &\hspace{3em}\quad + \frac12 q(q-1) \int_0^t
    \norm[\big]{\sigma'(u)(\overline{v}_\lambda - \overline{v}_\mu)
      B}^2_{\gamma(U;L^q_xH)}
    \norm[\big]{\overline{v}_\lambda - \overline{v}_\mu}^{q-2}_{L^q_xH}\,ds.
  \end{split}
  \end{equation}
  To estimate the first term on the right-hand side, let us set
  \[
  J^F_\lambda := \frac{1}{1-\lambda F},
  \]
  so that $F_\lambda = FJ_\lambda$ and
  \begin{gather*}
  F_\lambda\overline{v}_\lambda - F_\mu\overline{v}_\mu =
  FJ^F_\lambda\overline{v}_\lambda - FJ^F_\mu\overline{v}_\mu,\\
  \begin{split}
    \overline{v}_\lambda - \overline{v}_\mu
    &= J^F_\lambda \overline{v}_\lambda - J^F_\mu \overline{v}_\mu
    + \overline{v}_\lambda - J^F_\lambda \overline{v}_\lambda
    - \bigl( \overline{v}_\mu - J^F_\mu \overline{v}_\mu \bigr),\\
    &= J^F_\lambda \overline{v}_\lambda - J^F_\mu \overline{v}_\mu
    + \lambda F_\lambda \overline{v}_\lambda - \mu F_\mu \overline{v}_\mu.
  \end{split}
  \end{gather*}
  Then, recalling that $\Phi'_q(\overline{v}_\lambda -
  \overline{v}_\mu) = q \norm{\overline{v}_\lambda -
    \overline{v}_\mu}^{q-2}_H (\overline{v}_\lambda -
  \overline{v}_\mu)$,
  \begin{align*}
  &\ip[\big]{F_\lambda\overline{v}_\lambda - F_\mu\overline{v}_\mu}%
    {\Phi'_q(\overline{v}_\lambda - \overline{v}_\mu)}\\
  &\hspace{3em}= q \ip[\big]{FJ^F_\lambda\overline{v}_\lambda %
    - FJ^F_\mu\overline{v}_\mu}{\norm{\overline{v}_\lambda %
    - \overline{v}_\mu}^{q-2}_H (J^F_\lambda \overline{v}_\lambda %
    - J^F_\mu \overline{v}_\mu)}\\
  &\hspace{3em}\quad + q \ip[\big]{F_\lambda\overline{v}_\lambda %
    - F_\mu\overline{v}_\mu}{\norm{\overline{v}_\lambda %
    - \overline{v}_\mu}^{q-2}_H \lambda F_\lambda \overline{v}_\lambda %
    - \mu F_\mu \overline{v}_\mu},
  \end{align*}
  where, as $F \leq 0$,
  \begin{align*}
  &\ip[\big]{FJ^F_\lambda\overline{v}_\lambda %
    - FJ^F_\mu\overline{v}_\mu}{\norm{\overline{v}_\lambda %
    - \overline{v}_\mu}^{q-2}_H (J^F_\lambda \overline{v}_\lambda %
    - J^F_\mu \overline{v}_\mu)}\\
  &\hspace{3em} = \int_G \norm{\overline{v}_\lambda %
    - \overline{v}_\mu}^{q-2}_H \ip[\big]{FJ^F_\lambda\overline{v}_\lambda %
    - FJ^F_\mu\overline{v}_\mu}{J^F_\lambda \overline{v}_\lambda %
    - J^F_\mu \overline{v}_\mu}_H\\
  &\hspace{3em} = \int_G F \norm[\big]{J^F_\lambda \overline{v}_\lambda %
    - J^F_\mu \overline{v}_\mu}^2_H %
    \norm[\big]{\overline{v}_\lambda - \overline{v}_\mu}^{q-2}_H \leq 0,
  \end{align*}
  and
  \begin{align*}
  &\ip[\big]{F_\lambda\overline{v}_\lambda %
    - F_\mu\overline{v}_\mu}{\norm{\overline{v}_\lambda %
    - \overline{v}_\mu}^{q-2}_H \lambda F_\lambda \overline{v}_\lambda %
    - \mu F_\mu \overline{v}_\mu}\\
  &\hspace{3em} = \int_G \norm{\overline{v}_\lambda %
    - \overline{v}_\mu}^{q-2}_H \ip[\big]{F_\lambda\overline{v}_\lambda %
    - F_\mu\overline{v}_\mu}{\lambda F_\lambda \overline{v}_\lambda %
    - \mu F_\mu \overline{v}_\mu}_H\\
  &\hspace{3em} \lesssim (\lambda+\mu) \int_G \norm{\overline{v}_\lambda %
    - \overline{v}_\mu}^{q-2}_H \bigl( \abs{F_\lambda}
    \norm{\overline{v}_\lambda}_H^2 + \abs{F_\mu} \norm{\overline{v}_\mu}_H^2
    \bigr)\\
  &\hspace{3em} \leq (\lambda+\mu) \int_G \abs{F} \bigl( %
    \norm{\overline{v}_\lambda}_H + \norm{\overline{v}_\mu}_H \bigr)^q\\
  &\hspace{3em} \leq (\lambda+\mu) \norm[\big]{F}_{L^\infty_x} \bigl(
    \norm{\overline{v}_\lambda}_{L^q_xH} + \norm{\overline{v}_\mu}_{L^q_xH}
    \bigr)^q.
  \end{align*}
  Therefore
  \[
  \int_0^t \ip[\big]{F_\lambda\overline{v}_\lambda%
    - F_\mu\overline{v}_\mu}{\Phi'_q(\overline{v}_\lambda%
    - \overline{v}_\mu)}\,ds \lesssim (\lambda+\mu) q \int_0^t
  \norm[\big]{F}_{L^\infty_x} \bigl(
  \norm{\overline{v}_\lambda}_{L^q_xH} +
  \norm{\overline{v}_\mu}_{L^q_xH} \bigr)^q\,ds
  \]
  and
  \begin{align*}
  &\abs[\bigg]{\int_0^t \ip[\big]{F_\lambda\overline{v}_\lambda%
    - F_\mu\overline{v}_\mu}{\Phi'_q(\overline{v}_\lambda%
    - \overline{v}_\mu)}\,ds}^{1/q}\\
  &\hspace{3em} \lesssim (\lambda+\mu)^{1/q} q^{1/q} \norm[\Big]{%
     \norm[\big]{F}^{1/q}_{L^\infty_x} \bigl(
     \norm{\overline{v}_\lambda}_{L^q_xH} +
     \norm{\overline{v}_\mu}_{L^q_xH} \bigr)}_{L^q(0,t)}.
  \end{align*}
  The remaining terms on the right-hand side of \eqref{eq:vblm} can be
  estimated similarly to the corresponding estimates in the proof of
  Proposition~\ref{prop:vl}. In particular, one has
  \begin{align*}
  &\frac12 q(q-1) \int_0^t
    \norm[\big]{\sigma'(u)(\overline{v}_\lambda - \overline{v}_\mu)
      B}^2_{\gamma(U;L^q_xH)}
    \norm[\big]{\overline{v}_\lambda - \overline{v}_\mu}^{q-2}\,ds\\
  &\hspace{3em} \leq \frac12 q(q-1) L_\sigma^2
    \norm[\big]{B}^2_{\gamma(U;X)} \int_0^t
    \norm[\big]{\overline{v}_\lambda - \overline{v}_\mu}^q_{L^q_xH}\,ds,
  \end{align*}
  hence the third term on the right-hand side of \eqref{eq:vblm}
  raised to power $1/q$ is dominated by
  \[
  \bigl( q(q-1)/2)^{1/q} L_\sigma^{2/q}
  \norm[\big]{B}^{2/q}_{\gamma(U;X)} \norm[\big]{\overline{v}_\lambda
    - \overline{v}_\mu}_{L^q(0,t;L^q_xH)}.
  \]
  Denoting by $M$ the stochastic integral on the right-hand side of
  \eqref{eq:vblm}, one has
  \[
  \norm[\big]{(M_T^*)^{1/q}}_{L^p(\Omega)} 
  = \norm[\big]{(M_T^*)}^{1/q}_{L^{p/q}(\Omega)} 
  \lesssim q^{1/q} \norm[\big]{\overline{v}_\lambda 
    - \overline{v}_\mu}_{L^p(\Omega;L^{2q}(0,t;L^q_xH))},
  \]
  with an implicit constant depending on the norm of $B$ in
  $\gamma(U;X)$ and $p$, among others. We are left with
  \begin{align*}
  \norm[\big]{\overline{v}_\lambda - \overline{v}_\mu}_{L^p(\Omega;C([0,t];L^q_xH))}
  &\lesssim (\lambda+\mu)^{p/q} \norm[\Big]{%
     \norm[\big]{F}^{1/q}_{L^\infty_x} \bigl(
     \norm{\overline{v}_\lambda}_{L^q_xH} +
     \norm{\overline{v}_\mu}_{L^q_xH} \bigr)}_{L^p(\Omega;L^q(0,t))}\\
  &\quad + \norm[\big]{\overline{v}_\lambda 
    - \overline{v}_\mu}_{L^p(\Omega;L^{2q}(0,t;L^q_xH))}
    + \norm[\big]{\overline{v}_\lambda 
    - \overline{v}_\mu}_{L^p(\Omega;L^q(0,t;L^q_xH))},
  \end{align*}
  so that, for $T_0$ sufficiently small,
  \[
  \norm[\big]{\overline{v}_\lambda - \overline{v}_\mu}_{L^p(\Omega;C([0,T_0];L^q_xH))}
  \lesssim (\lambda+\mu)^{p/q} \norm[\Big]{%
     \norm[\big]{F}^{1/q}_{L^\infty_x} \bigl(
     \norm{\overline{v}_\lambda}_{L^q_xH} +
     \norm{\overline{v}_\mu}_{L^q_xH} \bigr)}_{L^p(\Omega;L^q(0,T_0))}.
  \]
  Iterating this reasoning over intervals of length $T_0$ covering
  $[0,T]$, we reach the conclusion that $(\overline{v}_\lambda)$ is a
  Cauchy net in $L^p(\Omega;C([0,T_0];L^q(G;H)))$ if
  \[
  \norm[\Big]{\norm[\big]{F}^{1/q}_{L^\infty_x}
    \norm{\overline{v}_\lambda}_{L^q_xH}}_{L^p(\Omega;L^q(0,T))}
  \]
  is bounded uniformly with respect to $\lambda$. From
  \[
  \norm[\Big]{\norm[\big]{F}^{1/q}_{L^\infty_x}
    \norm{\overline{v}_\lambda}_{L^q_xH}}_{L^q(0,T)} \leq
  \norm[\big]{F}^{1/q}_{L^\infty_{t,x}}
  \norm[\big]{\overline{v}_\lambda}_{L^q_{t,x}H},
  \]
  it follows that, for any $r,s > p$ such that $1/p = 1/r + 1/s$,
  \[
  \norm[\Big]{\norm[\big]{F}^{1/q}_{L^\infty_x}
    \norm{\overline{v}_\lambda}_{L^q_xH}}_{\bL^p L^q_t} \leq
  \norm[\big]{\overline{v}_\lambda}_{\bL^r L^q_{t,x}H}
  \norm[\big]{F}^{1/q}_{\bL^{s/q} L^\infty_{t,x}}.
  \]
  As already seen, $F = f(u)$ belongs to
  $L^p(\Omega;C([0,T;C(\overline{G})))$ for every $p>0$, hence one
  only has to show that $(\overline{v}_\lambda)$ is bounded in
  $L^r(\Omega;L^q(0,T;L^q(G;H)))$ for some $r>p$. But this can be
  obtained exactly as in the proof of Proposition~\ref{prop:vl}.
\end{proof}

The Cauchy property just proved, coupled with the regularizing
properties of the semigroup $S$, allow to obtain strong regularity
properties of the process $Du$.
\begin{thm}
  The process $Du$ belongs to $L^p(\Omega;C([0,T];L^\infty(G;H)))$ for
  every $p>0$, after modification on a subset of $\Omega \times [0,T]$
  of measure zero.
\end{thm}
\begin{proof}
  Let $\zeta$ be the (strong) limit in $L^p(\Omega;C([0,T];L^q(G;H)))$
  of the Cauchy sequence $(\overline{v}_\lambda)$. Passing to the
  limit as $\lambda to 0$ in \eqref{eq:vbl}, one easily sees that
  $\zeta$ solves \eqref{eq:v}, hence $\zeta=v=Du$. We are going to
  improve the regularity of $v$ using the regularizing properties of
  $S$. In particular, Lemma~\ref{lm:DPKZ} yields
  \[
    \norm[\big]{S \ast Fv}_{C([0,T];E^q_\eta(H)} \lesssim
    T^\varepsilon \norm[\big]{Fv}_{L^r(0,T;L^q(H)}
  \]
  for every $r>1$ and $0 \leq \eta < 1-1/r$, with $\varepsilon$ a
  positive constant. Therefore, taking $r=2p/(p+2)$, so that
  $1/r=1/2-1/p$, and $\eta>d/(2q)$, one has
  \[
    \norm[\big]{S \ast Fv}_{\bL^p C_t L^\infty_xH}
    \lesssim_T \norm[\big]{Fv}_{L^p(\Omega;L^r(0,T;L^q(G;H))},
  \]
  where
  \[
    \norm[\big]{Fv}_{L^r_t L^q_xH} \leq \norm[\big]{F}_{L_{t,x}^\infty}
    \norm[\big]{v}_{L^r_t L^q_xH}.
  \]
  Let $s,s' \in \erre_+$ be such that $1/p=1/s+1/s'$. Then
  \[
    \norm[\big]{S \ast (Fv)}_{\bL^p C_t L^\infty_xH} \lesssim_T
    \norm[\big]{F}_{\bL^{s'} L_{t,x}^\infty}
    \norm[\big]{v}_{\bL^s L^r_t L^q_xH},
  \]
  where, as already mentioned several times, both norms on the
  right-hand side are finite.
  The analogous estimate for the stochastic convolution term is
  similar (actually a bit simpler, as $\sigma'(u)$ is bounded): by
  Lemma~\ref{lm:msc}, taking $\alpha<1/2$ such that $\eta<\alpha-1/p$,
  one has
  \[
    \norm[\big]{S \diamond \sigma'(u) v B}_{\bL^p C_t E^q_\eta(H)}
    \lesssim_{p,T} \biggl(\int_0^T
    \E\norm[\big]{(t-\cdot)^{-\alpha}%
      \sigma'(u) v B}^p_{\gamma(L^2(0,t;U);L^q_xH)}\,dt \biggr)^{1/p},
  \]
  where
  \begin{align*}
    \norm[\big]{(t-\cdot)^{-\alpha} \sigma'(u) v B}_{\gamma(L^2(0,t;U);L^q_xH)}
    &\leq \norm[\big]{(t-\cdot)^{-\alpha} \sigma'(u) %
      v B}_{L^2(0,t;\gamma(U;L^q_xH))}\\
    &\leq L_\sigma \norm{B}_{\gamma(U;X)} \biggl( \int_0^t (t-s)^{-2\alpha}
      \norm{v(s)}^2_{L^q_xH}\,ds \biggr)^{1/2}\\
    &\lesssim_T L_\sigma \norm{B}_{\gamma(U;X)} \norm{v}_{C_tL^q_xH},
  \end{align*}
  hence
  \[
    \norm[\big]{S \diamond \sigma'(u) v B}_{\bL^p C_t L^\infty_xH}
    \lesssim_{p,T} L_\sigma \norm{B}_{\gamma(U;X)}
    \norm[\big]{v}_{\bL^p C_t L^q_xH},
  \]
  where the right-hand side is certainly finite.
\end{proof}


\section{Estimates with positivity-preserving covariance}
Recall that $L^2_Q$ is the completion of $L^2(G)$ with respect to the
norm ${\norm{\cdot}}_Q := {\norm{Q^{1/2}f}}_{L^2}$. Throughout this
section we shall assume that the bounded operator $Q$ on $L^2$ is
positivity preserving, i.e. that $Qg \geq 0$ if $g \in L^2$,
$g \geq 0$. Moreover, without loss of generality, we assume that
$\sigma \geq 0$.

Let $f_\lambda$ be the Yosida approximation of $f$. As already seen,
the Lipschitz continuity of $f_\lambda$ implies that
\[
  du_\lambda + Au_\lambda\,dt = f_\lambda(u_\lambda)\,dt + \sigma(u_\lambda)B\,dW,
  \qquad u_\lambda(0)=u_0 \in C(\overline{G})
\]
admits a unique mild solution
$u_\lambda \in L^p(\Omega;C([0,T];C(\overline{G})))$
for all $p \geq 1$, and $v_\lambda:=Du_\lambda$ satisfies \eqref{eq:vl}.

Let us introduce the auxiliary equation
\[
  y_\lambda(t) = v_{0,\lambda}(t)
  + \int_0^t S(t-s) \sigma'(u_\lambda(s)) y_\lambda(s) B\,dW(s).
\]
Both this equation and \eqref{eq:vl} are well-posed in
$L^p(\Omega;C([0,T];L^\infty(G;H)))$.

The following comparison result is the main tool to achieve
boundedness of the collection $(v_\lambda)$.
\begin{prop}
  \label{prop:vlyl}
  One has $\norm{v_\lambda(t,x)}_H \leq \norm{y_\lambda(t,x)}_H$ for
  almost all $(t,x) \in [0,T] \times G$.
\end{prop}
We proceed in several steps.
\begin{lemma}
  Let $h \in H$ be such that $Qh \geq 0$. Then $\ip{v_{0\lambda}}{h}
  \geq 0$.
\end{lemma}
\begin{proof}
  Note that $h(t) \in L^2_Q$ for a.a. $t \in [0,T]$, hence $Qh(t) \in
  L^2$ for a.a. $t \in [0,T]$ and $Qh \in L^2(0,T;L^2) \simeq
  L^2(G_T)$, so the statement $Qh \geq 0$ is meaningful.
  One has
  \begin{align*}
    \ip[\big]{v_{0\lambda}(t,x)}{h}_{H}
    &= \int_0^T\!\!\int_G K_{t-\tau}(x,z) \, \sigma(u_\lambda(\tau,z))
      \, 1_{[0,t]}(\tau) [Qh](\tau,z)\,d\tau\,dz\\
    &= \int_0^t S(t-\tau) \sigma(u_\lambda(\tau)) [Qh](\tau)\,d\tau,
  \end{align*}
  where $\sigma \geq 0$ and $S$ is a positivity-preserving
  semigroup. The result then follows immediately.
\end{proof}

Let us set, for any $h \in H$, $v_\lambda^h :=
\ip{v_\lambda}{h}_H$. Then $v_\lambda^h$ satisfies
\[
  v_\lambda^h(t) = \ip[\big]{v_{0,\lambda}}{h}_H
  + \int_0^t S(t-s) f'_\lambda(u_\lambda(s)) v_\lambda^h(s)
  + \int_0^t S(t-s) \sigma'(u_\lambda(s)) v_\lambda^h(s) B\,dW,
\]
i.e., by the previous lemma, it is the mild solution to
\begin{equation}
  \label{eq:lam}
  dv_\lambda^h + Av_\lambda^h = \bigl( \sigma(u_\lambda)Qh +
  f'_\lambda(u_\lambda) v_\lambda^h\bigr)\,dt + \sigma'(u_\lambda)
  v_\lambda^h B\,dW, \qquad v_\lambda^h(0) = 0.
\end{equation}
Completely analogously, $y_\lambda^h := \ip{y_\lambda}{h}_H$ ie the
mild solution to
\[
  dy_\lambda^h + Ay_\lambda^h = \bigl( \sigma(u_\lambda)Qh
  + \sigma'(u_\lambda) y_\lambda^h B\,dW, \qquad y_\lambda^h(0) = 0.
\]
We are going to compare $v_\lambda^h$ and $y_\lambda^h$ pointwise for
a certain class of vectors $h$. To do so, we need to impose a
regularity assumption on the noise that will be removed later.
\begin{lemma}
  Assume that $B \in \gamma(U;X)$, with $X$ a Banach space
  continuously embedded in $L^\infty$. If $h \in H$ is such that
  $Qh \geq 0$, then
  \begin{equation}
    0 \leq {\ip{v_\lambda}{h}}_H \leq {\ip{y_\lambda}{h}}_H.
  \end{equation}
\end{lemma}
\begin{proof}
  The boundedness of $f'_\lambda$ and the hypothesis on $B$ imply that
  the equation for $v^h_\lambda$ is well-posed in
  $L^p(\Omega;C([0,T];L^2))$, hence we can apply the maximum principle
  in \cite{cm:pos1}. This says that if
  \[
    -\ip[\big]{\sigma(u_\lambda)Qh + f'_\lambda(u_\lambda)\phi}{\phi^-}_{L^2}
    + \frac12 \norm[\big]{1_{\{\phi \leq 0\}} \sigma'(u_\lambda) \phi B}^2_{\cL^2}
    \lesssim \norm[\big]{\phi^-}^2_{L^2}
  \]
  for every $\phi \in L^2$, then $v_\lambda^h \geq 0$. Since
  $\sigma \geq 0$, $Qh \geq 0$, and $f'_\lambda$ is bounded by
  $1/\lambda$, one obtains
  \[
    -\ip[\big]{\sigma(u_\lambda)Qh + f'_\lambda(u_\lambda)\phi}{\phi^-}_{L^2}
    \leq \frac{1}{\lambda}
    \norm[\big]{\phi^-}^2_{L^2}.
  \]
  Moreover, thanks to the hypothesis on $B$ and the boundedness of
  $\sigma'$, $\phi \mapsto \sigma'(u_\lambda) \phi B$ is Lipschitz
  continuous with values in the space of Hilbert-Schmidt operators,
  and
  \[
    \norm[\big]{1_{\{\phi \leq 0\}} \sigma'(u_\lambda) \phi
      B}^2_{\cL^2} \lesssim \norm[\big]{\phi^-}^2_{L^2}.
  \]
  The proof that $v_\lambda^h \geq 0$ is thus completed. The
  difference $y^h_\lambda - v^h_\lambda$ satisfies
  \[
    y^h_\lambda(t) - v^h_\lambda(t)
    = \int_0^t S(t-s) (-f'_\lambda(u_\lambda))v^h_\lambda\,ds
    + \int_0^t S(t-s) \sigma'(u_\lambda) (y^h_\lambda(t) - v^h_\lambda(t))
      B\,dW(s),
  \]
  hence, again applying the above-mentioned comparison principle,
  $y^h_\lambda - v^h_\lambda \geq 0$ if
  \[
    \ip[\big]{f'_\lambda(u_\lambda)v_\lambda^h}{\phi^-}_{L^2}
    + \frac12 \norm[\big]{1_{\{\phi \leq 0\}} \sigma'(u_\lambda) \phi B}^2_{\cL^2}
    \lesssim \norm[\big]{\phi^-}^2_{L^2},
  \]
  which is the case because $f'_\lambda(u_\lambda) \leq 0$ and
  $v^h_\lambda \geq 0$. The proof is thus concluded.
\end{proof}

In order to remove the assumption on $B$ of the lemma, consider a
linear equation of the type
\[
  w(t) = w_0(t) + \int_0^t S(t-s) F(s)w(s)\,ds
  + \int_0^t S(t-s) \Sigma(s) w(s) C\,dW(s)
\]
for $L^q$-valued processes, where $F$ and $\Sigma$ are bounded
random fields. Then results on continuous dependence of solutions on
coefficients (cf.~\cite{KvN2}), or a direct computation using the
maximal estimate of Lemma~\ref{lm:msc}, shows that the map
\begin{align*}
  \gamma(U;L^q) &\longrightarrow L^p(\Omega;C([0,T];E_\eta)\\
  C &\longmapsto w
\end{align*}
is continuous.

Now we can prove the crucial estimate.
\begin{prop}
  If $h \in H$ is such that $Qh \geq 0$, then
  \begin{equation}
    \label{eq:ovi}
    0 \leq {\ip{v_\lambda}{h}}_H \leq {\ip{y_\lambda}{h}}_H.
  \end{equation}
\end{prop}
\begin{proof}
  Let $\alpha>0$ be such that $(I+A)^{-\alpha}L^q \embed L^\infty$ and
  set $B_\varepsilon := (I + \varepsilon A)^{-\alpha}B$,
  $\varepsilon>0$. Then $B_\varepsilon$ satisfies the hypothesis of
  the lemma and $B_\varepsilon \to B$ in $\gamma(U;L^q)$ as
  $\varepsilon \to 0$. Denote the solutions to the equations for
  $v^h_\lambda$ and $y^h_\lambda$ with $B$ replaced by $B_\varepsilon$
  by $v^h_{\lambda,\varepsilon}$ and $y^h_{\lambda,\varepsilon}$,
  respectively. Then
  $0 \leq v^h_{\lambda,\varepsilon} \leq y^h_{\lambda,\varepsilon}$ by
  the lemma, hence $0 \leq v^h_\lambda \leq y^h_\lambda$ taking the
  limit as $\varepsilon \to 0$.
\end{proof}

If $h \in L^2([0,T] \times G)$, $h \geq 0$, since we have assumed that
$Q$ is positivity preserving, then $Qh \geq 0$ and
\[
  0 \leq {\ip{v_\lambda}{h}}_H = \ip{Q^{1/2}v_\lambda}{Q^{1/2}h}
  = \ip{Qv_\lambda}{h}.
\]
Since this holds for an arbitrary such $h$, we infer that
\[
  Qy_\lambda(t,x) \geq Qv_\lambda(t,x) \geq 0 \qquad \text{for
    a.a. }(t,x) \in [0,T] \times G.
\]
In view of \eqref{eq:ovi}, we can thus proceed as follows:
$Qv_\lambda \geq 0$ yields
\[
\norm{v_\lambda}_H^2 = {\ip{v_\lambda}{v_\lambda}}_H
\leq {\ip{y_\lambda}{v_\lambda}}_H
\]
and $Qy_\lambda \geq 0$ yields
\[
  {\ip{v_\lambda}{y_\lambda}}_H \leq {\ip{y_\lambda}{y_\lambda}}_H
  = \norm{y_\lambda}_H^2.
\]
We have thus shown that
\[
{\norm{v_\lambda}}_H^2 \leq {\norm{y_\lambda}}_H^2,
\]
i.e. the proof of Proposition~\ref{prop:vlyl} is completed.

Since estimates in $L^p(\Omega;C([0,T];L^\infty)$ for
$\norm{y_\lambda(t,x)}_H$ uniform with respect to $\lambda$ can be
easily obtained, as $y_\lambda$ satisfies an equation with Lipschitz
coefficients, we arrive at the following result.
\begin{thm}
  Assume that $Q$ is positivity preserving. Then
  \[
  \sup_{(t,x)} \E\norm{Du(t,x)}_H^p < \infty
  \]
  for every $p>0$.
\end{thm}

This is the same result obtained in \cite{cm:POTA13}, as far as
first-order Malliavin derivatives goes, in the much simpler case of
additive noise, under the same conditions on $f$, $\sigma$, and
$Q$. Here we needed to assume slightly more on the semigroup $S$ in
order to extend it from $L^q$ to $L^q(H)$.


\bibliographystyle{amsplain}
\bibliography{ref}


\bigskip\bigskip

\setlength{\parindent}{0pt}

Carlo Marinelli\\
Department of Mathematics\\
University College London\\
Gower Street\\
London WC1E 6BT\\
United Kingdom\\
URL: \url{http://goo.gl/4GKJP}

\end{document}